\documentclass[12pt]{amsart}

\usepackage{latexsym,amsmath,amssymb,amsthm,amsfonts}

\usepackage[numeric]{amsrefs}

\usepackage{hyperref}

\usepackage[capitalise]{cleveref}

\usepackage{multido,pst-plot,pstricks,pst-node}
\newcommand{\Z}{{\mathbb{Z}}}
\renewcommand{\S}{{\mathbb{S}}}
\newcommand{\eps}{\varepsilon}

\newcommand{\E}{{\mathbb{E}}}

\renewcommand{\P}{{\mathbb{P}}}

\newcommand{\Vol}{{\rm Vol}}
\newcommand{\diam}{{\rm diam}\,}

\newcommand{\dist}{{\rm dist}}

\newcommand{\Y}{{\mathcal{Y}}}
\newcommand{\UU}{{\mathcal{U}}}

\newcommand{\bn}{{B_n}}

\renewcommand{\phi}{{{\varphi}}}

\newcommand{\qtq}[1]{{\quad\text{#1}\quad}}

\newtheorem{theorem}{Theorem}
\newtheorem{lemma}[theorem]{Lemma}

\newtheorem{corollary}[theorem]{Corollary}

\theoremstyle{remark}
\newtheorem{remark}[theorem]{Remark}

\title{On asymptotic Lebesgue's universal covering problem}
\author{A. Arman}
 \address{Department of Mathematics, University of Manitoba, Winnipeg, MB, R3T 2N2, Canada}
 \email{andrew0arman@gmail.com}
 \thanks{The first author was supported in part by a postdoctoral fellowship of the Pacific Institute for the Mathematical Sciences.}

\author{A.\ Bondarenko}
 \address{Department of Mathematical Sciences, Norwegian University of Science and
 	Technology, NO-7491 Trondheim, Norway}
 \email{andriybond@gmail.com}
 \thanks{The second author was supported in part by Grant 334466 of the Research Council of Norway.}

\author{A.\ Prymak}
 \address{Department of Mathematics, University of Manitoba, Winnipeg, MB, R3T 2N2, Canada}
 \email{prymak@gmail.com}
 \thanks{The third author was supported by NSERC of Canada Discovery Grant RGPIN-2020-05357.}

\author{D.\ Radchenko}
 \address{Universit\'{e} de Lille, CNRS, Laboratoire Paul Painlev\'{e}, F-59655 Villeneuve d'Ascq, France}
 \email{danradchenko@gmail.com}
\thanks{The fourth author was supported by ERC Starting Grant No. 101078782.}

\keywords{Universal cover, volume, Jung's theorem, Lebesgue's problem, diameter, Borsuk's problem}

\subjclass[2020]{Primary 52C17; Secondary 52A20, 52A22, 52A38, 52A39, 52A40, 49Q20}

\begin{document}
\begin{abstract}
A universal cover in $\mathbb{E}^{n}$ is a measurable set that contains a congruent copy of any set of diameter 1. Lebesgue's universal covering problem, posed in 1914, asks for the convex set of smallest area that serves as a universal cover in the plane ($n=2$).

A simple universal cover in $\mathbb{E}^n$ is provided by the classical theorem of Jung, which states that any set of diameter 1 in an $n$-dimensional Euclidean space is contained in a ball $J_n$ of radius $\sqrt{\tfrac{n}{2n+2}}$; in other words, $J_n$ is a universal cover in $\mathbb{E}^n$.

We show that in high dimensions, Jung's ball $J_n$ is asymptotically optimal with respect to the volume, namely, for any universal cover $U \subset \mathbb{E}^n$,
$$
{\rm Vol}(U) \ge \exp\bigl(-\sqrt{\tfrac54\,n \log n}-O(\sqrt{n/\log n})\,\bigr) \,{\rm Vol}(J_n).
$$
\end{abstract}
\maketitle

\section{Introduction and main result}
\label{sec:intro}

\subsection{Asymptotic bound on the volume of a universal cover}

A measurable set $U \subset \mathbb{E}^n$ is called a \emph{universal cover} if it contains a congruent
copy of every set $A \subset \mathbb{E}^n$ of diameter $1$.
Lebesgue's universal covering problem asks for the convex universal cover of smallest possible area in the plane ($n=2$). Originally posed by Lebesgue in a 1914 letter to P\'al and first published by P\'al in~\cite{Pa}, the problem has since been the focus of much activity, leading to a sequence of improvements both in lower bounds and explicit constructions. Important contributions include those of Sprague~\cite{Sp}, Hansen~\cite{Ha}, Elekes~\cite{El}, and Baez--Bagdasaryan--Gibbs~\cite{BBG}. 

Despite this progress, the problem remains open: the best current lower bound, $0.832$, is due to Brass and Sharifi~\cite{Br-Sh}, while the smallest known convex universal cover, of area $0.8440935944$, is due to Gibbs~\cite{Gi} (preprint). Although a non-convex cover of smaller area was obtained by Duff~\cite{Du}, it is now surpassed by the most recent convex construction~\cite{Gi}. Universal covers are closely linked to a range of classical extremal problems in convex and discrete geometry, with a notable connection to Borsuk's question on decomposing a set into parts of strictly smaller diameter.

The broader question of existence, structure, and optimality of universal covers in various dimensions has also inspired substantial work, especially in the context of polyhedral coverings. Notable contributions include a series of papers by Makeev~\cites{makeev:81,makeev:82,makeev:89,makeev:90,makeev:97}, as well as results of Kovalev~\cite{Ko}, Kuperberg~\cite{Kuperberg1999}, and Hausel--Makai--Szûcs~\cite{HMS}. For further perspectives, related problems, and survey material, see~\cite{BMP}*{Sect.~11.4} and~\cite{CFG}*{Section~D15}.

We denote the unit ball in $\mathbb{E}^n$ by $\bn$, and by $\Vol(\cdot)$ the Lebesgue measure in $\mathbb{E}^n$.
A classical theorem of Jung~\cite{Ju} states that the Euclidean ball $J_n = r_n \bn$ with radius 
$r_n := \sqrt{\frac{n}{2n+2}}$ is a universal cover.
Our main result shows that, in high dimensions, the volume of any universal cover is at least $(1 - o(1))^n \Vol(J_n)$,
so Jung's ball is an asymptotically optimal universal cover, up to a subexponential factor.

\begin{theorem}\label{thm:volumebound}
	Let $U$ be a universal cover in $\mathbb{E}^n$. Then
	\[
	\Vol(U) \ge \exp\bigl(-\sqrt{\tfrac54\,n \log n}-O(\sqrt{n/\log n})\,\bigr)\Vol(J_n)\,,
	\]
	where $J_n$ is the ball of radius $\sqrt{\frac{n}{2n+2}}$.
\end{theorem}

\begin{remark}
Assuming $U$ is convex, minor modifications to the proof yield a tighter bound with constant $\tfrac34$ instead of $\tfrac54$.
\end{remark}


\subsection{Relation to Borsuk's question.}\label{sec:borsuk}
Let $b(n)$ denote the smallest number $b$ such that any bounded set in $\mathbb{E}^n$ can be partitioned into $b$ subsets of strictly smaller diameter.
The example of a regular simplex shows that $b(n) \ge n+1$.
In 1933, Borsuk~\cite{Borsuk1933} asked whether $b(n) = n+1$ holds in all dimensions.
This was verified for $n=2,3$, but, unexpectedly, the conjecture fails in high dimensions:
in 1993, Kahn and Kalai~\cite{KaKa} proved that $b(n) > 1.2^{\sqrt{n}}$ for all sufficiently large~$n$.
The problem of determining $b(n)$, along with many related questions, remains of significant interest; see, for example, the survey by Kalai~\cite{Ka}.

In small dimensions, universal covers can be used to derive good estimates on $b(n)$ by partitioning a universal cover into pieces of diameter $<1$.
For example, the regular hexagon whose opposite sides are at distance~$1$ is a universal cover in the plane (see P\'al~\cite{Pa}), giving $b(2)\le 3$.
Similarly, suitably truncated octahedra were used by Eggleston~\cite{Eggleston1955} and Gr\"unbaum~\cite{Grunbaum1957} to obtain $b(3)\le 4$.
For general $n$, Lassak~\cite{Lassak1982} observed that for a unit vector $u$ the set $J_n\cap (r_nu+B_n)$ is a universal cover and used this to show
$b(n)\le 2^{\,n-1}+1$, which remains the best known upper bound for $4\le n\le 17$.

For high dimensions, Theorem~\ref{thm:volumebound} implies that universal covers must have volume at least $(1-o(1))^n(\tfrac1{\sqrt{2}})^n \Vol(\bn)$.
Hence, if a universal cover in $\mathbb{E}^n$ could be partitioned into $M$ subsets of diameter at most~$1$, then 
using the isodiametric inequality~\cite{EG}*{Thm. 2.4} and simple volume estimates one would obtain
$M \ge (\sqrt{2}-o(1))^n$.
This is weaker than the currently best known asymptotic bound $b(n)\le (\sqrt{3/2}+o(1))^n$ established by Schramm~\cite{Sch} and by Bourgain and Lindenstrauss~\cite{BoLi}.
Thus, one cannot hope to improve the estimate on $b(n)$ by partitioning universal covers into subsets of diameter strictly less than~$1$.

\subsection{Asymptotic optimality for other geometric measures}
A sharp result analogous to \cref{thm:volumebound}, but restricted to \emph{translative} universal covers, was obtained independently by Bezdek and Connelly~\cite{Be:mw-trans} and by Makeev~\cite{makeev:90}*{Th.~1}, who proved that the minimal value of mean width of a translative universal cover is exactly the mean width of Jung's ball $J_n$. For general congruent (not necessarily translative) universal covers, $J_n$ is not an exact minimizer for either volume or mean width, since, for example, for a unit $u\in\mathbb{E}^n$ the set $J_n\cap (r_nu+B_n)$ is a universal cover with both smaller volume and smaller mean width than $J_n$.

As a corollary of \cref{thm:volumebound}, we deduce that $J_n$ has asymptotically minimal mean width among all \emph{convex} universal covers: by Urysohn's inequality (see, e.g.~\cite{AGA1}*{Sect.~1.5.5}), any convex universal cover in $\mathbb{E}^n$ must have mean width at least $(1 - o(1))$ times that of $J_n$. Similarly, Alexandrov's inequalities (see, e.g.~\cite{AGA1}*{Sect.~1.1.5}) imply the asymptotic optimality of $J_n$ with respect to quermassintegrals and intrinsic volumes, including surface area.


\subsection{Outline of the proof of \cref{thm:volumebound}}
Assume, to the contrary, that $U$ is a minimal universal cover whose volume is too small (smaller than claimed by Theorem~\ref{thm:volumebound}). 
For an appropriately chosen radius $r$ very close to $r_n=\sqrt{\frac{n}{2n+2}}$, we will construct a discrete set $X \subset r \bn$ of diameter at most~$1$ which is not contained in any member of the family $\mathcal{U}$ consisting of all congruent copies of $U$ that intersect $r \bn$.
Although the family $\mathcal{U}$ is infinite, it admits a discretization in the following sense: there exists a finite family $\mathcal{U}_\varepsilon$ of congruent copies of a small ``thickening'' $U + \varepsilon \bn$ such that every member of $\mathcal{U}$ is contained in some member of $\mathcal{U}_\varepsilon$; see \cref{lem:cov-family}. So, it remains to construct the desired $X$ which is not contained in any member of $\mathcal{U}_\varepsilon$. We choose $\eps$ depending only on $n$ (namely, $\eps = \frac{1}{55\cdot 5^n}$) and obtain that the resulting cardinality of $\mathcal{U}_\varepsilon$ is quite large (bounded above by $7^{n^3}$), but sufficiently controlled for our argument.

The set $X$ is then chosen probabilistically: we take a suitable number of random points in $r \bn$, and, if necessary, delete certain points to ensure that the resulting set $X$ has diameter~$\le 1$. 
The probability that such a construction produces a set not covered by a fixed member of $\mathcal{U}_\varepsilon$ can be estimated in terms of $\Vol(U + \varepsilon \bn)$, which is very close to $\Vol(U)$; see \cref{sec:thick}. 
If $\Vol(U)$ is too small, and since the cardinality of $\mathcal{U}_\varepsilon$ is manageable,  then such a set $X$ must exist. 

In \cref{sec:meas_graphs} we develop a general framework for these probabilistic constructions, providing in \cref{lem:covering-abstract} a lower bound on covering numbers by elements of a discrete family in terms of measurable graphs. 
This framework generalizes the deletion method previously used in~\cite{ABP} and can be of independent interest. 
Its application to our geometric setting is carried out in \cref{sec:volumebound}.

\section{Proofs}

\subsection{Notations}\label{sec:notations}
Let $\mu$ be the rotation-invariant probability measure on $\S^{n-1}$, the unit sphere in
$\E^n$, and let
\[C(x,\alpha):=\{y\in\S^{n-1}:x\cdot y\ge \cos \alpha\}\]
be the spherical cap of radius $\alpha$ centred at $x\in\S^{n-1}$. We denote the measure
of the spherical cap of radius $\alpha$ by
$m(\alpha):=\mu(C(x,\alpha))$. By \cite{Bo-Wi}*{Cor.~3.2}, for $\alpha\in(0,\pi/2)$ we have
\begin{equation} \label{eq:capsize}
	\frac{\sin^{n-1}\alpha }{\sqrt{2\pi n}} < m(\alpha) < \frac{\sin^{n-1}\alpha}{\sqrt{2\pi (n-1)}\cos\alpha}.
\end{equation}
We use $|\cdot|$ to denote the cardinality of a finite set, while $\|\cdot\|$ denotes the Euclidean
norm.

\subsection{Lemma about $\eps$-thickening of a minimal universal cover}\label{sec:thick}

In each dimension, universal covers of minimal volume exist by~\cite{makeev:90}*{Th.~2}. To prove Theorem~\ref{thm:volumebound} we need the following property.

\begin{lemma}\label{lem:vol-diff-univ-cover}
	There is $n_0$ such that for any $n>n_0$, any universal cover $U$ of minimal volume in $\E^n$ and any $0<\eps<\frac1{55}$ we have
	\begin{equation}\label{eqn:univ-cover-volume}
		\frac{\Vol(U+\eps \bn)}{\Vol(U)} \le 1 + \eps\cdot 55\cdot 5^n\,.
	\end{equation}
\end{lemma}
While an estimate of the type~\eqref{eqn:univ-cover-volume} is natural to expect, $U$ is not assumed to be convex which causes certain additional difficulties. We also remark that only a rough estimate is needed and the result does not contribute significantly in \cref{thm:volumebound}.

Informally, the proof of \cref{lem:vol-diff-univ-cover} proceeds by reducing the
estimate for the volume of the $\eps$-neighbourhood of $U$ to estimating the total volume of a controlled collection of
one-dimensional enlargements. 
We first choose a finite set of directions $\rho$ which is sufficiently
well-spread on the sphere, but whose size is still not too large (at most $5^n$). 
For each such direction, whenever a line parallel to $\rho$ contains an interval
of length bounded below by an absolute constant, lying inside $U$, we
enlarge this interval slightly, by an amount of order $\eps$. 
The one-dimensional estimate in \cref{lem:one-dim-measure} shows that, for every such direction $\rho$, these
enlargements add only a small amount of volume, proportional to $\eps\Vol(U)$. 

On the other hand, since $U$ is a union of bodies of constant width~$1$, its
pieces are sufficiently ``roundish'': near every point one can find a cone of
definite aperture and length contained in the corresponding piece. 
By the covering property of the chosen directions $\rho$ and the geometric observation
in \cref{lem:cone}, the enlarged sets obtained in these directions cover
the whole set $(U+\eps\bn)\setminus U$, proving~\eqref{eqn:univ-cover-volume}.

\begin{lemma}\label{lem:one-dim-measure}
	For any measurable $U\subset\E^n$, direction $\rho\in \S^{n-1}$, and scalars $0<h_1<h_2$, the set
	\[
	T(\rho,h_1,h_2,U):=\{x\in\E^n: x+t\rho\in U\text{ for all }\, t\in[h_1,h_2]\}
	\]
	satisfies
	\[
	\Vol(T(\rho,h_1,h_2,U)\setminus U)\le \frac{h_1}{h_2-h_1}\Vol(U).
	\]
\end{lemma}
\begin{proof}
	It suffices to consider the case $n=1$, as for larger $n$ the inequality follows from the one-dimensional case by considering restrictions to the lines with the direction $\rho$. Now assuming $n=1$ and $\rho=1$, define $Q=\{x\in\partial U:[x,x+h_2-h_1]\subset U\}$. If $Q$ is infinite, then $\Vol(U)=\infty$ and the inequality is trivial. Otherwise, observe that
	\[
	(T(1,h_1,h_2,U)\setminus U)\subset \bigcup_{x\in Q}[x-h_1,x],
	\]
	so
	\[
	\Vol(T(\rho,h_1,h_2,U)\setminus U)\le h_1
	|Q|=\frac{h_1}{h_2-h_1}\Vol\left(\bigcup_{x\in Q}[x,x+h_2-h_1]\right)
	\le  \frac{h_1}{h_2-h_1}\Vol(U). \qedhere
	\]
\end{proof}

Let $K(a,\xi,\alpha,\ell)$ be the cone with apex $a\in\E^n$, axis direction
$\xi\in\S^{n-1}$, height $\ell>0$, and half-angle $\alpha\in(0,\pi/2)$:
\[
K(a,\xi,\alpha,\ell):=\{x\in\E^n:\|x-a\|\cos\alpha\le (x-a)\cdot \xi\le \ell\}.
\]
\begin{lemma}\label{lem:cone}
	Suppose $\alpha\in(0,\pi/2)$ and $\ell>0$. For constants $\eps_0=\frac{\ell}{3}\tan\frac{\alpha}{2}\cos\alpha,\; c_1=\frac{1}{\sin (\alpha/2)},\; c_2=\frac{\ell}{3}$, for any $a\in\E^n$, $\xi\in\S^{n-1}$, $\rho\in C(\xi,\alpha/2)$, $0<\eps<\eps_0$, we have
	\[
	a+\eps \bn\subset T(\rho,c_1\eps,c_1\eps+c_2,K(a,\xi,\alpha,\ell)).
	\]
\end{lemma}
\begin{proof}
    For any $0<\eps<\eps_0$,
	let $K_\eps:=K(a,\xi,\tfrac\alpha2,\tfrac23\ell)\setminus
	K(a,\xi,\tfrac\alpha2,\eps\cot\tfrac\alpha2)$ be the corresponding ``truncated
	cone'' having half the angle of that of $K(a,\xi,\alpha,\ell)$, the same axis and
	direction. It is elementary to check that $\dist(K_\eps,\partial
	K(a,\xi,\alpha,\ell))=\eps$. The required inclusion follows provided $a+t\rho\in
	K_\eps$ whenever $\rho\in C(\xi,\alpha/2)$ and $c_1\eps\le t\le c_1\eps+c_2$.
	Projecting onto the direction of $\xi$, we observe that for $\rho\in C(\xi,\alpha/2)$ we have $a+t\rho\in K_\eps$ if
	$\eps \cot\tfrac{\alpha}{2}\le t(\rho\cdot\xi)\le \tfrac23 \ell$, in particular, if $c_1\eps\le t\le c_1\eps+c_2$ (note $c_1\eps<\tfrac \ell3$ by the choice of $\eps_0$). 
\end{proof}

\begin{proof}[Proof of \cref{lem:vol-diff-univ-cover}.]
	It is well known that for any body $W$ of constant width 1, the largest
	inscribed ball and the smallest circumscribed ball have the same centre and the
	sum of their radii is $1$, see~\cite{Maetal}*{Th.~3.4.1, p.~69}. By Jung's theorem~\cite{Ju}, the radius of the circumball of $W$
	is less than $R:=\frac1{\sqrt{2}}$, so the radius of the inscribed ball is at
	least $r:=1-\frac1{\sqrt2}$. Thus, by convexity of $W$, for any $x\in W$, one can find
	$\xi\in\S^{n-1}$ with $K(x,\xi,\alpha,r)\subset W$ and $\alpha=\arcsin\tfrac
	rR$. As $\sin(\alpha/2)>1/5$, by standard sphere covering arguments (\cite{Bo-Wi}*{Th.~1.1}, \eqref{eq:capsize}), for
	sufficiently large $n_0$ we can find a set $X\subset\S^{n-1}$ of $N=5^n$
	directions (independent of $W$) such that any spherical cap of radius
	$\alpha/2$ contains a direction from~$X$. Therefore, we can apply
	\cref{lem:cone} (with $\alpha=\arcsin\frac{r}{R}$ and $\ell=r$) to obtain the following statement: for any body of constant
	width~$1$ in~$\E^n$, any $x\in W$, and any $0<\eps<\eps_0$, there exists
	$\rho\in X$ such that
	\[
	x+\eps \bn \subset T(\rho,c_1\eps,c_1\eps+c_2,K(x,\xi,\alpha,r))\subseteq T(\rho,c_1\eps,c_1\eps+c_2,W).
	\]
	Since any minimal universal cover $U$ is the union of a family of bodies of constant width $1$, the above statement is true when $W$ is replaced with $U$. Now we  apply \cref{lem:one-dim-measure} for each $\rho\in X$ and obtain
	\[
	\Vol((U+\eps \bn)\setminus U)\le N \frac{c_1\eps}{c_2} \Vol(U),
	\]
	which concludes the proof after the constants are tracked.
\end{proof}

\subsection{General lower bound on covering in terms of measurable graphs}
\label{sec:meas_graphs}

We now give a general probabilistic construction of cocliques in graphs that are
hard to cover by sets from a given finite family. A graph $\Gamma$ is a pair
$(V,E)$ where $V$ is the set of vertices and $E\subseteq V\times V$ is the set of edges of $\Gamma$. We always
assume that $\Gamma$ is unoriented, that is, $(x,y)\in E$ if and only if $(y,x)\in
E$. We note that in our applications $V$ is infinite. A vertex subset $X\subseteq V$ is a coclique if $(x,y)\not\in E$ for all $x,y\in X$. We will also denote by $\Delta(V)$ the diagonal $\{(v,v)\colon v\in
V\}\subseteq
V\times V$.

Given a set $\Omega$, a family $\Y$ of subsets of $\Omega$, and $X\subset \Omega$, we
define the covering number as
\[N(X,\Y):=\min\bigg\{|\UU|:X\subset \bigcup_{Y\in\UU}Y,\, \UU\subset\Y\bigg\}.\]
In other words, $N(X, \Y)$ is the minimal number of sets from $\Y$ needed to cover $X$.

Our key auxiliary result is the following theorem. 
\begin{theorem}\label{lem:covering-abstract}
Let $\Gamma=(V,E)$ be a graph and assume that $V$ is equipped with a probability measure~$\nu$ satisfying $(\nu\times\nu)(\Delta(V))=0$. Let $\Y$ be a finite family
of measurable subsets of $V$. Assume $0<p<\tfrac12$, $M,k\in \Z_{>0}$, and $1\le k\le \frac{1}{2p}$. If
\begin{gather}
\nu(Y)\le p \qtq{for all} Y\in \Y, \label{eqn:size each Y}\\
 |\Y|\le
\frac{1}{2}({2ekp})^{-\frac{M}{2k}}, \qtq{and} \label{eqn:size family Y}\\
(\nu\times\nu)(E)\le\frac{1}{2M}, \label{eqn:edges bound}
\end{gather}
then there exists a coclique $X\subset V$ such that
\[
|X|\ge M/2 \qtq{and} N(X,\Y) > k\,.
\]
\end{theorem}

Informally, the theorem asserts that one can construct a sufficiently large coclique $X$ in a measurable graph $\Gamma$ that is difficult to cover by members of~$\Y$. 
The sets in $\Y$ play the role of potential ``covering pieces'', and the assumptions require only that each such set is not too ``large''~\eqref{eqn:size each Y} and that the total number of these sets is controlled in a relatively ``mild'' way~\eqref{eqn:size family Y}. 
The edges of $\Gamma$ encode an ``undesirable'' relation between vertices~--- one that we want to avoid in our final set~$X$. 
Thus, in order for a large coclique to exist, we require that the graph does not contain ``too many'' edges~\eqref{eqn:edges bound}. 
The construction itself is probabilistic: we begin by randomly selecting $M$ vertices of $\Gamma$, and then delete at most $M/2$ of them to remove all edges among the chosen vertices, leaving a coclique of the desired size.

\begin{proof}[Proof of~\cref{lem:covering-abstract}.]
Let $z_1,\dots,z_M$ be i.i.d. random variables in $V$ distributed with respect to
$\nu$,
and let $Z=\{z_1,\dots,z_M\}$. Since $\Delta(V)$ has $\nu\times\nu$ measure $0$, we have $|Z|=M$
with probability~1.

Next, for any $Y\in\Y$, the cardinality $|Z\cap Y|$ has a binomial distribution
${\rm Bi}(M, \tilde p)$ where $\tilde p = \nu(Y)\le p$. Applying Chernoff's
bound~\cite{Versh}*{Th.~2.3.1}, we have
	\[\P(|Z\cap Y|\ge \tfrac{M}{2k}) < (2ekp)^{\frac{M}{2k}}\,.\]
Therefore, by the union bound, with probability $>1/2$ we have
$|Z\cap Y| < \tfrac{M}{2k}$ for all $Y\in\Y$.

Finally, since $\P((z_i,z_j)\in E) \le \frac{1}{2M}$ for all $i<j$, the expected number of
edges in $Z$ is at most $\binom{M}{2}\frac{1}{2M}<\frac{M}{4}$. Then by
Markov's inequality with probability $>1/2$ the number of edges in $Z$ is
at most $M/2$, and so removing one vertex from each edge gives a coclique $X$ with $\ge
M/2$
elements and with nonzero probability it satisfies $|X\cap Y|<\frac{M}{2k}$ for all
$Y\in\Y$, completing the proof.
\end{proof}

\subsection{Volume bound for universal covers}
\label{sec:volumebound}
Now we apply the general bound for our geometric setting. Recall (see \cref{sec:notations}) that $C(v,\alpha)$ is the spherical cap of radius $\alpha$ centred at $v\in\S^{n-1}$, and $m(\alpha)$ is the measure of such a cap.

\begin{lemma}\label{lem:low-covering-geom}
Let $0<\alpha<\pi/2$, $0<p<\tfrac{1}{2k}$, $k\in \mathbb{Z}_{>0}$, and suppose that $\Y$ is a family of measurable sets in $\E^n$
such that
\[
\Vol(Y\cap r \bn)\le p \Vol(r \bn)\quad\text{for all}\quad Y\in \Y.
\]
If $N(X,\Y)\leq k$ for every subset $X\subset r \bn$ with $\diam(X)<2r\cos(\alpha/2)$, then \[
|\Y| > \frac{1}{2}(2ekp)^{-1/(4m(\alpha)k)}.
\] 
\end{lemma}
\begin{proof}
	After rescaling, we may assume that $r=1$. We are going to apply
	\cref{lem:covering-abstract}
	with $V= \bn$, $E=\{(x,y)\in V^2\colon \|x-y\|\ge
	2\cos(\alpha/2)\}$, and $\nu(A)=\frac{\Vol(A\cap  \bn)}{\Vol( \bn)}$.
	It remains to estimate the measure of $E$. We claim that
	for any $x\in  \bn$
	\begin{equation}\label{eq:volume_balls}
		\nu(\{y\in  \bn:\|x-y\|\ge 2\cos(\alpha/2)\})\le m(\alpha).
	\end{equation}
	Note that $2\cos(\alpha/2)>\sqrt2$, so~\eqref{eq:volume_balls} is trivial when $x=0$.
	Assuming $x\ne0$, set $v:=-x/\|x\|$. Then for any $y\in  \bn$ satisfying $\|x-y\|\ge
	2\cos(\alpha/2)$, recalling that $x\in\bn$, we have
	\[
	-x\cdot y \ge 2\cos^2(\alpha/2)-\frac{\|x\|^2+\|y\|^2}{2} \ge \cos\alpha\ge
	\|x\|\cos\alpha,
	\]
	i.e., $v\cdot y\ge \cos\alpha$. This means that such $y$ must belong to the convex
	hull of $C(v,\alpha)$, hence, to the convex hull $K$ of $C(v,\alpha)$ and the origin.
	It remains to observe that $\Vol(K)=m(\alpha)\Vol( \bn)$,
	and~\eqref{eq:volume_balls} follows. This clearly implies that
	$(\nu\times\nu)(E)\le
	m(\alpha)$.
	Now taking $M=\frac{1}{2m(\alpha)}$ in \cref{lem:covering-abstract}
	we complete the proof.
\end{proof}

Next we need a construction of an
appropriate
\emph{finite} family of sets covering any congruent copy of a given set in $\E^n$. For $\eps>0$ and bounded $X\subset\E^n$, we denote the $\eps$-covering number of $X$ by
$N(X,\eps):=N(X,\mathcal{B}_{\eps})$, where
$\mathcal{B}_{\eps}$ is the collection of all closed balls of radius $\eps$ in $\E^n$. In other words,
$N(X,\eps)$ denotes the size of the smallest $\eps$-covering of $X$.
\begin{lemma}\label{lem:cov-family}
    Suppose $D, \eps>0$ and $V\subset\E^n$ is bounded. Then there exists a finite set $T$ of isometries of $\E^n$ with
	\[
    |T|\le 2N(V,\tfrac{\eps}{2D})(\tfrac{500D}{\eps})^{\frac{n(n-1)}{2}},
    \]
    such that for any compact set $K\subset\E^n$ of diameter $D$ containing the origin and any congruent copy $K'=AK+v$ with $A\in O(n)$, $v \in V$, we have $K'\subseteq g(K_{\eps})$ for some $g\in T$, where $K_\eps=K+\eps \bn$ is the $\eps$-thickening of $K$.
\end{lemma}
\begin{proof}
Let $G$ be the isometry group of $\E^n$, with the metric $d(f,g)=\max_{\|x\|\le 1}\|f(x)-g(x)\|$. Note that for $f,g\in G$ we have $f(K)\subseteq g(K_{\eps})$ provided $\rho(f(K),g(K))\le \eps$, where~$\rho$ is the Hausdorff metric on compact subsets of~$\E^n$.  
Since $K$ is contained in the ball of radius $D$, we have $\rho(f(K),g(K))\le D\cdot d(f,g)$. Note that any $f\in G$ can be uniquely written as $f(x)=Ax+v$ with $v=f(0)\in\E^n$ and $A\in O(n)$. Therefore, it suffices to take $T$ to be an $(\eps/D)$-covering in the set $G'=\{f\in G\colon f(0)\in V\}$.

For $f(x)=Ax+v$ and $g(x)=Bx+w$, where $A,B\in O(n)$ and $v,w\in\E^n$, we have $d(f,g) \le \|A-B\|_{op}+\|v-w\|$ with
$\|\cdot\|_{op}$ denoting the operator norm on matrices. By~\cite{Sza}*{Prop.~6} (see also~\cite{Sza}*{Cor.~4} for derivation of the constants) any minimal
$\eps$-net in $O(n)$ has size at most $2(c/\eps)^{\frac{n(n-1)}{2}}$, where $c=3\pi e^{\pi}<220$.
Taking $T$ to consist of $f(x)=Ax+v$ with~$A$ running over an $\frac{\eps}{2D}$-net in $O(n)$ and~$v$ over an $\frac{\eps}{2D}$-covering in~$V$, we get the required bound on $|T|$.
\end{proof}

\begin{corollary}
    \label{cor:ball}
    Under the hypotheses of \cref{lem:cov-family}, if $V$ is a ball of radius $R>1$ and $\frac{\eps}{D}<1$, then one can find $T$ with $|T|\le
R^n(\tfrac{500D}{\eps})^{\frac{n(n+1)}{2}}$ satisfying the conclusion of the lemma.
\end{corollary}
\begin{proof}
    Taking a maximal $\tfrac{\eps}{2D}$-separated set in $V$ and comparing volumes gives $N(V,\frac{\eps}{2D})\le \left(\frac{R+\frac{\eps}{4D}}{\frac{\eps}{4D}}\right)^n$, which implies the corollary.
\end{proof}


We also need a simple bound on the diameter of a minimal volume universal cover. 
\begin{lemma}\label{lem:diamU}
If $U$ is a universal cover in $\E^n$ of minimal volume, then $\diam(U)\le 2\cdot 5^n$.
\end{lemma}
\begin{proof}
    By minimality, $U$ is a connected set which is the union of a family of bodies of constant width in $\E^n$. If $D=\diam(U)$, we can find at least $k:=\lceil \frac{D}2\rceil$ parallel hyperplanes $H_1,\dots,H_k$ such that each $H_i$ intersects $U$ while the distance between any $H_i$ and $H_j$, $i\ne j$, is at least $2$. Let $K_i\subseteq U$ be a body of constant width having a common point with $H_i$, $1\le i\le k$. Then, as $\diam(K_i)\le 1$, the sets $K_i$ have mutually disjoint interiors. Jung's theorem implies $\Vol((1-\tfrac1{\sqrt{2}})\bn)\le \Vol(K_i)$ and $\Vol(U)\le \Vol(\sqrt{2}\bn)$, so $k\Vol((1-\tfrac1{\sqrt{2}})\bn)\le \Vol(\sqrt{2}\bn)$ giving the stated bound.
\end{proof}

\begin{proof}[Proof of Theorem~\ref{thm:volumebound}]
	Let $n\ge n_0$, and $U$ be a universal cover in $\E^n$ of minimal volume containing the origin, which exists by~\cite{makeev:90}*{Th.~2}. Assume to the contrary that $\Vol(U)$ is ``too small'', i.e. $p=\frac{\Vol(U)}{\Vol(J_n)}<\exp\left(-\sqrt{\tfrac54n\log n}-c\sqrt{n/\log n}\right)$ for an absolute $c>0$ which will be chosen later.
	By Lemma~\ref{lem:vol-diff-univ-cover}, with $\eps = \frac{1}{55\cdot 5^n}$, the volume of the $\eps$-thickening $U_\eps=U+\eps \bn$ of $U$ satisfies
	\begin{equation}
	    \label{eqn:volu2}
        \Vol(U_{\eps}) \le 2\Vol(U) \,.
	\end{equation}
	As $n\ge n_0$, by \cref{lem:diamU} $\diam (U)< 6^n-\frac{1}{\sqrt{2}}$. Thus, applying \cref{cor:ball} with $D=6^n$ and $V$	a ball of radius $6^n$, gives us a finite family $T$ of isometries,
	with $|T|\le
	6^{\frac{n^2(n+3)}{2}}(\tfrac{500}{\eps})^{\frac{n(n+1)}{2}}<\frac127^{n^3}$,
	such that every isometric copy $AU+v$ (with $A\in O(n)$ and $v\in V$) is contained in $g(U_{\eps})$ for some $g\in T$. In particular, any isometric copy of $U$ intersecting $\tfrac{1}{\sqrt{2}} \bn$ is of the form $AU+v$ (with $A\in O(n)$ and appropriate $v\in V$), and so is 
	contained in $g(U_{\eps})$ for some $g\in T$. 
    
    Therefore, since $U$ is a universal cover, any set $X\subset r_n\bn$ of diameter $<1$ is covered by an element of $\Y=\{g(U_{\eps}) \colon g\in T\}$, and $|\Y|=|T|\leq \frac{1}{2}7^{n^3}$.  
    
    We are now in a position to apply Lemma~\ref{lem:low-covering-geom}. We take $r<r_n$, $\cos(\alpha/2)=\frac{1}{2r}$, $k=1$, with $\alpha<\pi/2$ to be chosen in a moment. Using~\eqref{eqn:volu2} and the definition of $p$, we have 
    $\Vol(U_\eps\cap r\bn)\leq \Vol(U_\eps)\leq 2p(r_n/r)^n \Vol(r\bn)$, and so for $n\ge n_0$
        \[
        \frac{1}{2}(4ep(r_n/r))^{-1/4m(\alpha)}\leq |\Y|<\frac{1}{2}7^{n^3},
        \]
        or
	\begin{equation} \label{eq:mainineq}
	4ep \ge (r/r_n)^n7^{-4m(\alpha)n^3}\,.
	\end{equation}
    We choose $\alpha$ so that $\sin(\alpha)=1-\frac{5\log n}{2n}$. Then a simple calculation shows that
	$\cos(\alpha)= \sqrt{\frac{5\log n}{n}}\left(1+O(\frac{\log n}{n})\right)$,  $\cos(\alpha/2)=\frac{1}{\sqrt{2}}\left(1+\sqrt{\frac{5 \log n}{4n}}+O(\frac{\log n}{n})\right)$, and eventually $(r_n/r)^n =
	\exp(\sqrt{(5/4)n\log n}+O(\log n))$. Now, using~\eqref{eq:capsize} we get
	$4m(\alpha)\leq \frac{c_1}{n^{5/2}\,\,\sqrt{\log n}}$ for appropriate $c_1>0$.
	Plugging these estimates into~\eqref{eq:mainineq} we obtain
		\[p \geq c_2 \exp\bigl(-\sqrt{(5/4)n\log n}-(\log 7)c_1 \sqrt{\tfrac{n}{\log n}}\bigr)\,,\]
	which, with a proper choice of $c$, is a contradiction proving the theorem. 
\end{proof}


{\bf Acknowledgement.} We are grateful to Gil Kalai for several insightful comments and questions on his blog, most notably comment~3 of~\cite{Ka-blog}, which helped stimulate this work. We would also like to thank Beatrice-Helen Vritsiou for asking us about covering a set of diameter $1$ by families of bodies of fixed diameter. Finally, the work benefited from the referee's valuable feedback.


\begin{bibsection}
		\begin{biblist}

\bib{ABP}{article}{
   author={Arman, Andrii},
   author={Bondarenko, Andrii},
   author={Prymak, Andriy},
   title={Convex bodies of constant width with exponential illumination
   number},
   journal={Discrete Comput. Geom.},
   volume={74},
   date={2025},
   number={1},
   pages={196--202},
}


\bib{AGA1}{book}{
   author={Artstein-Avidan, Shiri},
   author={Giannopoulos, Apostolos},
   author={Milman, Vitali D.},
   title={Asymptotic geometric analysis. Part I},
   series={Mathematical Surveys and Monographs},
   volume={202},
   publisher={American Mathematical Society, Providence, RI},
   date={2015},
   pages={xx+451},
   isbn={978-1-4704-2193-9},
}

\bib{BBG}{article}{
   author={Baez, John C.},
   author={Bagdasaryan, Karine},
   author={Gibbs, Philip},
   title={The Lebesgue universal covering problem},
   journal={J. Comput. Geom.},
   volume={6},
   date={2015},
   number={1},
   pages={288--299},
}

    \bib{Be:mw-trans}{article}{
   author={Bezdek, K\'aroly},
   author={Connelly, Robert},
   title={The minimum mean width translation cover for sets of diameter one},
   journal={Beitr\"age Algebra Geom.},
   volume={39},
   date={1998},
   number={2},
   pages={473--479},
}		


			\bib{Bo-Wi}{article}{
				author={B\"{o}r\"{o}czky, K., Jr.},
				author={Wintsche, G.},
				title={Covering the sphere by equal spherical balls},
				conference={
					title={Discrete and computational geometry},
				},
				book={
					series={Algorithms Combin.},
					volume={25},
					publisher={Springer, Berlin},
				},
				isbn={3-540-00371-1},
				date={2003},
				pages={235--251},
			}

\bib{Borsuk1933}{article}{
author={Borsuk, Karol},
title={Drei Sätze über die n-dimensionale euklidische Sphäre},
journal={Fundamenta Mathematicae},
volume={20},
number={1},
date={1933},
pages={177--190},
}

\bib{BoLi}{article}{
   author={Bourgain, J.},
   author={Lindenstrauss, J.},
   title={On covering a set in ${\bf R}^N$ by balls of the same diameter},
   conference={
      title={Geometric aspects of functional analysis (1989--90)},
   },
   book={
      series={Lecture Notes in Math.},
      volume={1469},
      publisher={Springer, Berlin},
   },
   date={1991},
   pages={138--144},
}


\bib{BMP}{book}{
   author={Brass, Peter},
   author={Moser, William},
   author={Pach, J\'anos},
   title={Research problems in discrete geometry},
   publisher={Springer, New York},
   date={2005},
   pages={xii+499},
   isbn={978-0387-23815-8},
   isbn={0-387-23815-8},
}

\bib{Br-Sh}{article}{
   author={Brass, Peter},
   author={Sharifi, Mehrbod},
   title={A lower bound for Lebesgue's universal cover problem},
   journal={Internat. J. Comput. Geom. Appl.},
   volume={15},
   date={2005},
   number={5},
   pages={537--544},
}





\bib{CFG}{book}{
   author={Croft, Hallard T.},
   author={Falconer, Kenneth J.},
   author={Guy, Richard K.},
   title={Unsolved problems in geometry},
   series={Problem Books in Mathematics},
   note={Corrected reprint of the 1991 original [MR1107516 (92c:52001)];
   Unsolved Problems in Intuitive Mathematics, II},
   publisher={Springer-Verlag, New York},
   date={1994},
   pages={xvi+198},
   isbn={0-387-97506-3},
}

\bib{Du}{article}{
   author={Duff, G. F. D.},
   title={A smaller universal cover for sets of unit diameter},
   journal={C. R. Math. Rep. Acad. Sci. Canada},
   volume={2},
   date={1980/81},
   number={1},
   pages={37--42},
}

\bib{Eggleston1955}{article}{
author={Eggleston, H.G.},
title={Covering a three-dimensional set with sets of smaller diameter},
journal={J. London Math. Soc.},
volume={30},
date={1955},
pages={11--24},
}




\bib{El}{article}{
  author={Elekes, G.},
  title={Generalized breadths, circular Cantor sets, and the least area UCC},
  journal={Discrete Comput. Geom.},
  volume={12},
  number={1},
  date={1994},
  pages={439--449},
}

\bib{EG}{book}{
   author={Evans, L. C.},
   author={Gariepy, R. F.},
   title={Measure theory and fine properties of functions},
   series={Textbooks in Mathematics},
   edition={2},
   publisher={CRC Press, Boca Raton, FL},
   date={2025},
   pages={xi+327},
   isbn={978-1-032-94644-3},
   isbn={978-1-003-58300-4},
   isbn={978-1-032-95055-6},
}
            


\bib{Gi}{article}{
title={An upper bound for Lebesgue's covering problem},
author={Gibbs, Philip},
date={2018},
eprint={https://arxiv.org/abs/1810.10089v1}
}

\bib{Grunbaum1957}{article}{
author={Grünbaum, Branko},
title={A simple proof of Borsuk's conjecture in three dimensions},
journal={Proc. Cambridge Philos. Soc.},
volume={53},
date={1957},
pages={776--778},
}


\bib{Ha}{article}{
  author={Hansen, H.},
  title={Small universal covers for sets of unit diameter},
  journal={Geom. Dedicata},
  volume={42},
  date={1992},
  pages={205--213},
}

\bib{HMS}{article}{
  author={Hausel, Tamas},
  author={Makai, Endre Jr.},
  author={Szûcs, András},
  title={Inscribing cubes and covering by rhombic dodecahedra via equivariant topology},
  journal={Mathematika},
  volume={47},
  number={1–2},
  date={2000},
  pages={371--397},
}

\bib{Ju}{article}{
   author={Jung, Heinrich},
   title={Ueber die kleinste Kugel, die eine r\"aumliche Figur einschliesst},
   language={German},
   journal={J. Reine Angew. Math.},
   volume={123},
   date={1901},
   pages={241--257},
}

\bib{KaKa}{article}{
    author = {Kahn, Jeff},
    author = {Kalai, Gil},
    title = {A counterexample to Borsuk's conjecture},
    journal = {Bull. Amer. Math. Soc. (New Series)},
    volume = {29},
    number = {1},
    date = {1993},
    pages = {60--62},
}


\bib{Ka}{incollection}{
  author={Kalai, Gil},
  title={Some old and new problems in combinatorial geometry I: Around Borsuk's problem},
  book={title={Surveys in Combinatorics 2015}},
  editor={Czumaj, Artur and Georgakopoulos, Agelos and Král, Daniel and Lozin, Vadim and Pikhurko, Oleg},
  publisher={Cambridge University Press},
  date={2015},
  pages={147–174},
}

\bib{Ka-blog}{article}{
  author = {Gil Kalai},
  title = {Progress Around Borsuk's Problem},
  eprint = {https://gilkalai.wordpress.com/2023/12/04/progress-around-borsuks-problem/},
  date = {Blog post, 4 December 2023}
}


\bib{Ko}{article}{
author={Kovalev, M.~D.},
title={The smallest Lebesgue covering exists},
journal={Mat. Zametki},
volume={40},
date={1986},
number={3},
pages={401--406, 430},
language={russian},
}

\bib{Kuperberg1999}{article}{
   author={Kuperberg, Greg},
   title={Circumscribing Constant-Width Bodies with Polytopes},
   journal={New York J. Math.},
   volume={5},
   date={1999},
   pages={91--100},
}


\bib{Lassak1982}{article}{
  author={Lassak, Marek},
  title={An estimate concerning Borsuk's partition problem},
  journal={Bull. Acad. Polon. Sci. Ser. Math.},
  volume={30},
  date={1982},
  pages={449--451},
}



			\bib{makeev:81}{article}{
				author={Makeev, V. V.},
				title={Universal coverings. I},
				language={russian},
				journal={Ukrain. Geom. Sb.},
				volume={24},
				date={1981},
				pages={70–79, iii},
			}

			\bib{makeev:82}{article}{
				author={Makeev, V. V.},
				title={Universal coverings. II},
				language={russian},
				journal={Ukrain. Geom. Sb.},
				volume={25},
				date={1982},
				pages={82--86, 142--143},
			}

            \bib{makeev:89}{article}{
   author={Makeev, V. V.},
   title={Universal coverings and projections of bodies of constant width},
   language={russian},
   journal={Ukrain. Geom. Sb.},
   date={1989},
   number={32},
   pages={84--88},
   translation={
      journal={J. Soviet Math.},
      volume={59},
      date={1992},
      number={2},
      pages={750--752},
   },
}


            \bib{makeev:90}{article}{
				author={Makeev, V. V.},
				title={Extremal coverings},
				language={russian},
				journal={Mat. Zametki},
				volume={47},
				date={1990},
				number={6},
				pages={62--66, 159},
				issn={0025-567X},
				translation={
					journal={Math. Notes},
					volume={47},
					date={1990},
					number={5-6},
					pages={570--572},
					issn={0001-4346},
				},
			}

\bib{makeev:97}{article}{
  author={Makeev, V. V.},
  title={On affine images of a rhombo-dodecahedron circumscribed about a three-dimensional convex body},
  language={russian},
  journal={Zap. Nauchn. Sem. S.-Peterburg. Otdel. Mat. Inst. Steklov. (POMI)},
  volume={246},
  date={1997},
  journal={Geom. i Topol.},
  number={2},
  pages={191--195, 200},
}
            

\bib{Maetal}{book}{
	author={Martini, H.},
	author={Montejano, L.},
	author={Oliveros, D.},
	title={Bodies of constant width},
	note={An introduction to convex geometry with applications},
	publisher={Birkh\"auser/Springer, Cham},
	date={2019},
	pages={xi+486},
	isbn={978-3-030-03866-3},
	isbn={978-3-030-03868-7},
}




\bib{Pa}{article}{
author={P\'al, J.},
title={\"Uber ein elementares Variationsproblem},
journal={Math.-fys. Medd., Danske Vid. Selsk.},
volume={3},
date={1920},
pages={35~pp}
}



\bib{Sch}{article}{
  author = {Schramm, O.},
  title = {Illuminating sets of constant width},
  journal = {Mathematika},
  volume = {35},
  date = {1988},
  number = {2},
  pages = {180--189},
}


\bib{Sp}{article}{
author={Sprague, R.},
title={\"Uber ein elementares Variationsproblem},
journal={Mat. Tidsskr. Ser. B},
date={1936},
pages={96--98}
}

			\bib{Sza}{inproceedings}{
				author = {Szarek, S. J.},
				title = {Nets of Grassmann manifold and orthogonal group},
				booktitle = {Proceedings of Banach Space Workshop},
				publisher = {University of Iowa Press},
				date = {1982},
				pages = {169--185}
			}

			\bib{Versh}{book}{
				author={Vershynin, R.},
				title={High-dimensional probability},
				series={Cambridge Series in Statistical and Probabilistic Mathematics},
				volume={47},
				note={An introduction with applications in data science;
					With a foreword by Sara van de Geer},
				publisher={Cambridge University Press, Cambridge},
				date={2018},
				pages={xiv+284},
			}

		\end{biblist}
	\end{bibsection}
\end{document}